\documentclass[11pt]{article}

\usepackage{enumerate}
\usepackage{amsmath}
\usepackage{amssymb,latexsym}
\usepackage{amsthm}
\usepackage{color}
\usepackage{graphicx,psfrag}
\usepackage{array}
\usepackage{hyperref,url}
\usepackage{amscd}

\title{Double blocking sets of size $3q-1$ in $\PG(2,q)$}
\author{Bence Csajb\'ok\thanks{MTA--ELTE Geometric and Algebraic Combinatorics Research Group, ELTE E\"otv\"os Lor\'and University, Budapest, Hungary, Department of Geometry, 1117, P\'azm\'any P.~s\'et\'any 1/C; email: csajbokb@cs.elte.hu.} and Tam\'as H\'eger\thanks{MTA--ELTE Geometric and Algebraic Combinatorics Research Group, ELTE E\"otv\"os Lor\'and University, Budapest, Hungary, Department of Computer Science, 1117, P\'azm\'any P.~s\'et\'any 1/C; email: hetamas@cs.elte.hu. Both authors were supported by OTKA Grant no.~K 124950 and the J\'anos Bolyai Research Scholarship of the Hungarian Academy of Sciences.}}
\date{}

\newcommand{\cB}{{\mathcal B}}
\newcommand{\cT}{{\mathcal T}}

\newcommand{\cL}{{\mathcal L}}

\newcommand{\cS}{{\mathcal S}}

\newcommand{\F}{{\mathbb F}}

\newcommand{\cut}[1]{}

\newtheorem{theorem}{Theorem}[section]

\newtheorem{lemma}[theorem]{Lemma}
\newtheorem{corollary}[theorem]{Corollary}
\newtheorem{definition}[theorem]{Definition}
\newtheorem{proposition}[theorem]{Proposition}

\newtheorem{conjecture}[theorem]{Conjecture}

\DeclareMathOperator{\PG}{{PG}}
\DeclareMathOperator{\AG}{{AG}}

\DeclareMathOperator{\PGL}{{PGL}}

\begin{document}
\date{\today}
\maketitle

\begin{abstract}
The main purpose of this paper is to find double blocking sets in $\PG(2,q)$ of size less than $3q$, in particular when $q$ is prime.
To this end, we study double blocking sets in $\PG(2,q)$ of size $3q-1$ admitting at least two $(q-1)$-secants. We derive some structural properties of these and show that they cannot have three $(q-1)$-secants. This yields that one cannot remove six points from a triangle, a double blocking set of size $3q$, and add five new points so that the resulting set is also a double blocking set.
Furthermore, we give constructions of minimal double blocking sets of size $3q-1$ in $\PG(2,q)$ for $q=13$, $16$, $19$, $25$, $27$, $31$, $37$ and $43$. If $q>13$ is a prime, these are the first examples of double blocking sets of size less than $3q$. These results resolve two conjectures of Raymond Hill from 1984.
\end{abstract}

\bigskip
{\it AMS subject classification:} 51E21

\bigskip
{\it Keywords:} double blocking set, finite projective plane.

\section{Introduction}

A $t$-fold blocking set of $\PG(2,q)$ is a set of points that intersects every line in at least $t$ points, and it is called minimal if none of its proper subsets is a $t$-fold blocking set. Usually, $1$-fold and $2$-fold blocking sets are called blocking sets and double blocking sets; $t$-fold blocking sets with $t\geq2$ are also called multiple blocking sets. Blocking and multiple blocking sets of finite projective planes are widely studied objects. A trivial lower bound for the size of a $t$-fold blocking set is $t(q+1)$. For detailed lower bounds, we refer the reader to \cite{BallMBS, BLSSz, GSzW}.

If $q$ is a square, one can easily construct a $t$-fold blocking set of size $t(q+\sqrt{q}+1)$ in $\PG(2,q)$ using the well-known partition of the pointset of $\PG(2,q)$ into Baer subplanes. This construction is the smallest possible if $t$ is small enough as shown in \cite{BSSz}. Up to our knowledge, surprisingly few constructions are known for small multiple blocking sets if $q$ is not a square. If $q$ is not a prime, \cite{BHSz,DBHSzVdV} give general constructions of small double blocking sets (of size around $2(q+(q-1)/(r-1))$, where $r$ is the order of a proper subfield of $\F_q$) as the union of two disjoint blocking sets. No other general results are known when $t$ is a constant. (For particular results on double blocking sets, see \cite{DGMP} and \cite{PS} as cited in \cite[p52]{BLSSz}.) If $q$ is a prime, the situation is even worse. 

A trivial construction for a double blocking set is the union of the sides of a triangle (that is, three non-concurrent lines), which is of size $3q$. In \cite{BB}, it was shown that a double blocking set in $\PG(2,q)$, $q\leq 8$, must have at least $3q$ points, and the question whether smaller examples may exist for larger values of $q$, $q$ prime, was left wide open. The first and, so far, only smaller example is shown in \cite{BKW}, where a double blocking set of size $3q-1$ for $q=13$ was constructed.

For this paragraph, let $q\in\{13,16,19,25,27,31,37,43\}$. In Section \ref{constructions} of the present paper, we show minimal double blocking sets in $\PG(2,q)$ of size $3q-1$ found by computer search. The complements of these are maximal $(q^2-2q+2,q-1)$-arcs in $\PG(2,q)$, which can be used to construct linear codes of type $[q^2-2q+2,3,q^2-3q+3]_q$ (see \cite{Hirsch}). If $13\neq q$ is a prime, then these are the first examples of such objects. Let us remark that our construction for $q=13$ is different from that of \cite{BKW}. The double blocking sets we present admit two $(q-1)$-secants, and their existence disprove a cautiously stated `conjecture' of Raymond Hill \cite{Hill1984} (see later). In Section \ref{structure}, some general structural properties of such double blocking sets are derived.
 
Hill considered the following problem \cite[Problem 3.8, p377]{Hill1984}: is it possible to delete $x$ points from a triangle and add $x-1$ points so that the resulting set of $3q-1$ points is a double blocking set? He proved that this is not possible for $x\leq 5$, and conjectured that it is also impossible for $x=6$ \cite[p378]{Hill1984}. Easy combinatorial arguments show, as pointed out in \cite{Hill1984}, that there are two options: a double blocking set obtained in this way (a) either contains a full line, or (b) the sides of the triangle become $(q-1)$-secants. We verify this conjecture in Section \ref{Hillconj} and prove the following theorem.

\begin{theorem}\label{Hillconjthm}
In $\PG(2,q)$, there is no double blocking set of size $3q-1$ that can be obtained by removing six points of a triangle and adding five new points.
\end{theorem}

Option (a) follows easily from the celebrated result on affine blocking sets due to Jamison and Brouwer--Schrijver. Our proof for option (b) is a somewhat laborious mixture of case analysis and tedious calculations. Hill proved the following theorem, which immediately yields that option (b) is not possible if $q\equiv 2 \pmod{3}$.

\begin{theorem}[{\cite[Theorem 3.10]{Hill1984}}]\label{Hill}
Suppose that $\cB$ is a double blocking set of size $3q-1$ with at least two $(q-1)$-secants in $\PG(2,q)$, $q>2$. Then $q\not\equiv 2\pmod{3}$.
\end{theorem}

On \cite[p380]{Hill1984} Hill remarks that ``The evidence for $q\leq 7$ suggests the conjecture that there does not exist a $(3q-1,\geq2)$-set with $r_{q-1}\geq 2$ for \underline{any} $q$ [that is, a double blocking set of size $3q-1$ having at least two $(q-1)$-secants]. The first cases for which such a set might exist are $q=9$ and $q=13$.'' The examples in Section \ref{constructions} refute this conjecture (the one in \cite{BKW} does not). Moreover, we propose the following

\begin{conjecture}
For all prime power $q\geq 13$, $q\not\equiv2 \pmod 3$, there exists a double blocking set in $\PG(2,q)$ of size $3q-1$ admitting two $(q-1)$-secants.
\end{conjecture}

Let us note that there is no such double blocking set for $q=9$ (a computer search quickly shows this).

\noindent\textbf{Preliminaries and notation.} $\PG(2,q)$ and $\AG(2,q)$ denote the projective and affine planes over $\F_q$, the finite field of order $q$, respectively. The multiplicative group of $\F_q$ will be denoted by $\F_q^\times$, and $\F_q^*$ stands for the set of non-zero elements of $\F_q$. To represent the points and lines of $\PG(2,q)$, we shall use homogeneous triplets in round brackets for points, considered as coloumn vectors, and in square brackets for lines, considered as row vectors. 
Recall that the coordinates of points and lines are defined up to a scalar multiplier, and $(x:y:z)\in [a:b:c]$ if and only if $ax+by+cz=0$. Usually we consider $\PG(2,q)$ as the closure of $\AG(2,q)$, where the additional line $\ell_\infty$ is called the line at infinity; clearly, we may assume $\ell_\infty=[0:0:1]$. For the points of $\ell_\infty$, we will sometimes use the notation $(m):=(1:m:0)$ ($m\in\F_q$) and $(\infty):=(0:1:0)$. 
The terms $X$ axis and $Y$ axis refer to the lines $[0:1:0]$ and $[1:0:0]$, which will be denoted by $L_X$ and $L_Y$, respectively. The slope of $[a:b:c]$ is $\infty$ if $b=0$ and $-a/b$ otherwise. Note that the slope of the line joining $(0:0:1)$ and $(x:y:z)$, $x\neq0$, is $y/x$.
With respect to a given pointset $S$, a $t$-secant is a line intersecting $S$ in precisely $t$ points. In case of $t=0$, 1, 2 and 3, a $t$-secant is also called a skew, tangent, bisecant or trisecant line (to $S$), respectively. A line is blocked by $S$ if it is not skew to $S$. We will frequently use the well-known fact that $\PGL(2,q)$, the group of projectivities of $\PG(2,q)$, is sharply transitive on the quadruples of points in general position and, dually, on the quadruples of lines in general position as well. Recall that if a projectivity $\varphi$ of $\PG(2,q)$ is represented by $M\in\F_q^{3\times3}$ (in notation, $\varphi=\langle M\rangle$), and the triplets $v$, $w$ represent the coordinates of a point and a line of $\PG(2,q)$, then their images under $\varphi$ are represented by $Mv$ and $wM^{-1}$, respectively.

\section{Properties of double blocking sets in $\PG(2,q)$ of size $3q-1$ with two $(q-1)$-secants}
\label{structure}

In this section we consider double blocking sets $\cB$ in $\PG(2,q)$ of size $3q-1$ admitting two $(q-1)$-secants. Let us remark that, as straightforward combinatorial arguments show, if $q\geq 7$ then the two $(q-1)$-secants of $\cB$ must intersect in a point of $\cB$; furthermore, if $q\geq 9$ and there are three $(q-1)$-secants to $\cB$ then they cannot be concurrent. As mentioned in the introduction, there are no double blocking sets of size less than $3q$ in $\PG(2,q)$ for $q\leq 8$ \cite{BB}. Thus, without loss of generality, we may assume that two $(q-1)$-secants of $\cB$ meet in a point of $\cB$, and if there are three $(q-1)$-secants to $\cB$, then they are not concurrent. Finally, let us note that a double blocking set having a $q$-secant clearly contains at least $3q$ points (we look around from the point of the $q$-secant not in the blocking set).

First we give the proof of Theorem \ref{Hill} in order to gain more detailed information from it. This proof is essentially the same as which was published in \cite{Hill1984}. Note that it might be regarded as a Segre-type argument (cf.~\cite{Segre}), but addition is used instead of multiplication. We start by formulating a lemma, whose assertion is essentially proved in \cite[Theorem 3.10]{Hill1984} but in a slightly different setting; this formulation is a bit more informative and helps to derive not only Theorem \ref{Hill} but further corollaries as well.

\textbf{Notation.} Let $L_X=[0:1:0]$ and $L_Y=[1:0:0]$ denote the $X$ and $Y$ axes, and let $X_\infty=(1:0:0)$, $X_1=(1:0:1)$, $Y_\infty=(0:1:0)$ and $Y_1=(0:1:1)$.

Applying a suitable projectivity, any double blocking set containing two $(q-1)$-secants and their point of intersection can be moved into the position described in the following lemma.

\begin{lemma}[see also \cite{Hill1984}]\label{trisecants}
Suppose that $\cB$ is a double blocking set of size $3q-1$, where all points of $L_X$ and $L_Y$ are in $\cB$ except for the points $X_1$, $X_\infty$, $Y_1$ and $Y_\infty$. 
Let $\cT$ be the set of lines through the origin that are different from the axes and intersect $\cB$ in more than two points. Then there exists $\mu,s\in\F_q^*$ such that a line through the origin is in $\cT$ if and only if its slope is $s$, $s\mu$ or $s\mu^2$, where $\mu^2+\mu+1=0$.
\end{lemma}
\begin{proof}
Let $\cS:=\cB \setminus (L_X \cup L_Y)=\{(x_i:y_i:z_i) \colon i=1,2,\ldots,q+2\}$. Note that $x_i\neq 0$ and $y_i\neq0$ for all $1\leq i\leq q+2$.

\begin{itemize}
\item[(A)] Looking at the points of $\cS$ from lines through $(0:1:0)$ it follows that the multiset $\{z_i/x_i \colon i=1,2,\ldots,q+2\}$ contains each
element of $\F_q$ once, except for 1 and 0, which are contained twice. In detail, the line joining $(0:1:0)$ and $(1:0:t)$ contains as many points of $\cS$ as the number of $i$s for which $\left|\begin{matrix}0&1&0\\1&0&t\\x_i&y_i&z_i\end{matrix}\right|=tx_i-z_i=0$ occurs, hence this number must be two if $t=0,1$ and one otherwise.

\item[(B)] Looking at the points of $\cS$ from lines through $(0:1:1)$ it follows that the multiset $\{(z_i-y_i)/x_i \colon i=1,2,\ldots,q+2\}$ contains each
element of $\F_q$ once, except for 1 and 0, which are contained twice. The reason is similar as above, $\left|\begin{matrix}0&1&1\\1&0&t\\x_i&y_i&z_i\end{matrix}\right|=tx_i+y_i-z_i=0$ must occur twice if $t=0,1$ and once otherwise.

\item[(C)] Looking at the points of $\cS$ from lines through $(0:0:1)$ it follows that the multiset $\{y_i/x_i \colon i=1,2,\ldots,q+2\}$ is contained in
$\F_q^*$ and it contains each element of $\F_q^*$ at least once. Clearly, there must be at least one point of $\cS$ on each line $[1:t:0]$, $t\neq 0$.

\item[(D)] Looking at the points of $\cS$ from lines through $(1:0:0)$ it follows that the multiset $\{z_i/y_i \colon i=1,2,\ldots,q+2\}$ contains each
element of $\F_q$ once, except for 1 and 0, which are contained twice. This follows from (A) by interchanging the first two coordinates.

\item[(E)] Looking at the points of $\cS$ from lines through $(1:0:1)$ it follows that the multiset $\{(z_i-x_i)/y_i \colon i=1,2,\ldots,q+2\}$ contains each
element of $\F_q$ once, except for 1 and 0, which are contained twice. This follows from (B) by interchanging the first two coordinates.
\end{itemize}

Since for $q\neq 2$ the sum of the elements of $\F_q$ is 0, the above observations yield several equalities.
From (A) we obtain $\sum_{i=1}^{q+2} z_i/x_i=1$, while (B) gives $\sum_{i=1}^{q+2} (z_i-y_i)/x_i=1$, and thus
\begin{equation}
\sum_{i=1}^{q+2} y_i/x_i=0.\label{AB}
\end{equation}
From (D) we obtain $\sum_{i=1}^{q+2} z_i/y_i=1$, while (E) gives $\sum_{i=1}^{q+2} (z_i-x_i)/y_i=1$, and thus
\begin{equation}
\sum_{i=1}^{q+2} x_i/y_i=0.\label{DE}
\end{equation}
Now we apply (C). Let $i,j,k\in \{1,2,\ldots, q+2\}$ such that $H:=\{x_\nu/y_\nu \colon \nu\notin \{i,j,k\}\}$ is the set of non-zero elements of $\F_q$. Note that if $\nu\in\{i,j,k\}$, then $s_\nu=x_\nu/y_\nu$ is the slope of a line through the origin which intersects $\cB$ in more than two points.
Clearly, $\sum_{h\in H} h = 0$ and $\sum_{h\in H} h^{-1}=0$, thus also $s_i+s_j+s_k=0$ by \eqref{AB} and $1/s_i+1/s_j+1/s_k=0$ by \eqref{DE}. 
Let $\mu=s_j/s_i$. Combining the last two equations we obtain $\mu^2+\mu+1=0$. As $s_i+s_j+s_k=0$, $1+s_j/s_i+s_k/s_i=0$ holds, whence $s_k/s_i=\mu^2$.
\end{proof}

Hill's Theorem \ref{Hill} follows immediately from the well-known fact that $\mu^2+\mu+1=0$ has a solution in $\F_q$ if and only if $q\not\equiv 2\pmod{3}$. Indeed, as $(\mu^3-1)=(\mu-1)(\mu^2+\mu+1)=0$, we have either $\mu=1$ and thus $3=0$, so $q\equiv0\pmod{3}$, or the order of $\mu$ in $\F_q^\times$ is three, whence $q \equiv 1 \pmod{3}$. However, the information gained on the `long secants' through the origin give valuable corollaries.

\begin{corollary}
\label{Osecants}
Let $\cB$ be a 2-fold blocking set of size $3q-1$ in $\PG(2,q)$, $q\neq 2$, such that $\cB$ has two $(q-1)$-secants, $\ell$ and $m$, $\ell\cap m\in\cB$. 
If $q\equiv 0 \pmod 3$, then through $\ell \cap m$ there pass two $(q-1)$-secants, $q-2$ bisecants and one 5-secant of $\cB$. 
If $q \equiv 1 \pmod 3$, then through $\ell \cap m$ there pass two $(q-1)$-secants, $q-4$ bisecants and three 3-secants of $\cB$. 
\end{corollary}
\begin{proof}
We may assume that $\ell$ and $m$ are as in Lemma \ref{trisecants} and then apply the lemma. If $q\equiv 0\pmod 3$, then $\mu=1$ and $|\cT|=1$. If $q\equiv 1\pmod 3$, then $\mu\neq 1$ and thus $|\cT|=3$. This finishes the proof.
\end{proof}

\begin{corollary}
\label{0mod3}
Let $\cB$ be a 2-fold blocking set of size $3q-1$ in $\PG(2,q)$, $q\neq 2$, such that $\cB$ has three $(q-1)$-secants. Then $q\equiv 1 \pmod 3$. 
\end{corollary}
\begin{proof}
Theorem \ref{Hill} excludes the case $q\equiv 2 \pmod 3$. Suppose now $q\equiv 0\pmod 3$. Let $\ell$, $m$ and $n$ be the three $(q-1)$-secants of $\cB$.
According to Corollary \ref{Osecants}, there is a 5-secant of $\cB$ at each of the verices of the triangle formed by $\ell$, $m$ and $n$. 
But then $\cB \setminus (\ell \cup m \cup n)$ has size at least 6, a contradiction since the size of this pointset is 5.
\end{proof}

\begin{definition}
Suppose that $\cB$ is a double blocking set of size $3q-1$, where all points of $L_X$ and $L_Y$ are in $\cB$ except for the points $X_1$, $X_\infty$, $Y_1$ and $Y_\infty$. Let $s$ denote the slope of a line through the origin, different from the axes, which intersects $\cB$ in more than two points. Then the \emph{parameter of $\cB$} is $s^3$.
\end{definition}

According to Lemma \ref{trisecants}, the parameter is well-defined both when $q\equiv 1\pmod 3$ or $q\equiv 0\pmod 3$.

Given a double blocking set $\cB$ in the setting of Lemma \ref{trisecants}, the projectivities of $\PG(2,q)$ that map $\cB$ to a double blocking set in the same setting are precisely those that permute the points $X_\infty, Y_\infty, X_1$ and $Y_1$ with the restriction that $\{X_1,X_\infty\}$ is either fixed setwise or is mapped to $\{Y_1,Y_\infty\}$. Let us denote the group of these projectivities by $G$. Then $G$ is isomorphic to the dihedral group $D_4$, and it is generated by the projectivities represented by
\[
F = 
\begin{pmatrix}
0 & 1 & 0 \\
-1 & 0 & 0 \\
0 & 1 & -1
\end{pmatrix}
\mbox{ and \;}
T = 
\begin{pmatrix}
0 & 1 & 0 \\
1 & 0 & 0 \\
0 & 0 & 1
\end{pmatrix},
\]
where $F$ maps $Y_\infty\to X_1\to Y_1 \to X_\infty\to Y_\infty$ ($F$ is an order four rotation in $D_4$), and $T$ is the reflection to the line $[1:-1:0]$. In this section, by rotations and reflections we will refer to the set of group elements $\{F^i\colon 1\leq i\leq 4\}$ and $\{TF^i\colon 1\leq i\leq 4\}$, respectively.

\begin{proposition}\label{relations}
Suppose that $\cB$ is a double blocking set of size $3q-1$ in $\PG(2,q)$, $q\neq2$, where all points of $L_X$ and $L_Y$ are in $\cB$ except for the points $X_1$, $X_\infty$, $Y_1$ and $Y_\infty$. Let $s^3$ be the parameter of $\cB$, and let 
$\cB\cap X_\infty Y_\infty=\{(1:m:0), (1:m':0)\}$,
$\cB\cap X_\infty Y_1=\{(x:1:1), (x':1:1)\}$,
$\cB\cap Y_\infty X_1=\{(1:y:1), (1:y':1)\}$,
$\cB\cap X_1Y_1=\{(1:a:a+1), (1:b:b+1)\}$.
Then $mm'=-s^3$, $xx'=1/s^3$, $yy'=s^3$, $ab=-s^3$.
\end{proposition}
\begin{proof}
Let $\cS'=\cB \setminus (L_X \cup L_Y \cup \ell_\infty)=\{(x_i:y_i:1) \colon i=1,2,\ldots,q\}$, and $\cS=\cS'\cup\{(1:m:0), (1:m':0)\}$. Note that $x_i\neq 0$ and $y_i\neq0$ for all $1\leq i\leq q$. Looking at the points of $\cS'$ from $(0:1:0)$ we see that the multiset $\{x_i\colon i=1,\ldots,q\}$ contains each element of $\F_q^*$ once except for 1, which is contained twice. Thus $\prod_{i=1}^q x_i=-1$ (recall Wilson's Theorem saying $\prod_{x\in\F_q^*}x=-1$). Similarly, $\prod_{i=1}^q y_i=-1$. Next we look at the points of $\cS$ from $(0:0:1)$. By Lemma \ref{trisecants}, we know the multiset of the slopes defined by the lines $OP$, $P\in\cS$: if $q\equiv 0\pmod 3$, then it contains each element of $\F_q^*$ once except for $s$, which is contained four times; if $q\equiv 1\pmod 3$, then it contains each element of $\F_q^*$ once except for $s$, $s\mu$, $s\mu^2$, which are contained twice. As $\mu^3=1$, in both cases we get
\[\prod_{i=1}^q \frac{y_i}{x_i}mm' = -s^3.\]
Since $\prod_{i=1}^q x_i=\prod_{i=1}^q y_i=-1$, we obtain $mm'=-s^3$. The other three assertions follow easily by applying the result to the images of $\cB$ under the projectivities of $G$ in the following way.
Consider \[F^2=\begin{pmatrix}
1 & 0 & 0 \\
0 & 1 & 0 \\
1 & 1 & -1
\end{pmatrix},\:
\overline{T}:=TF^3=\begin{pmatrix}
1 & 0 & 0 \\
0 & -1 & 0 \\
1 & 0 & -1
\end{pmatrix},\mbox{\, and \:}
F = 
\begin{pmatrix}
0  & 1 & 0 \\
-1 & 0 & 0 \\
0  & 1 & -1
\end{pmatrix}.
\]
By $\varphi$ we will always denote the projectivity represented by one of them. Let $(s')^3$ be the parameter of $\cB'$, the image of $\cB$ under $\varphi$. 

Let now $\varphi=\langle F^2\rangle$. A line $[a:b:0]$ through the origin is mapped to $[a:b:0]F^2=[a:b:0]$, hence $(s')^3=s^3$. 
On the other hand, a point $(1:x:x+1)$ is mapped to $F^2(1:x:x+1) = (1:x:0)$, hence $\cB'\cap\ell_\infty = \{(1:a:0),(1:b:0)\}$, and thus $ab=-s^3$.
Let $\varphi=\langle \overline{T}\rangle$. A line $[a:b:0]$ through the origin is mapped to $[a:b:0]\overline{T}=[a:-b:0]$, which yields that $(s')^3=-s^3$. 
On the other hand, a point $(1:t:1)$ is mapped to $\overline{T}(1:t:1) = (1:-t:0)$, hence $\cB'\cap\ell_\infty = \{(1:-y:0),(1:-y':0)\}$, and thus $(-y)(-y')=yy'=-(s')^3=s^3$.
Finally, let $\varphi=\langle F\rangle$. A line $[a:b:0]$ is mapped to $[a:b:0]F^3=[-b:a:0]$, which yields that $(s')^3=-1/s^3$. 
On the other hand, a point $(t:1:1)$ is mapped to $F(t:1:1) = (1:-t:0)$, hence $\cB'\cap\ell_\infty = \{(1:-x:0),(1:-x':0)\}$, and thus $(-x)(-x')=xx'=-(s')^3=1/s^3$.
\end{proof}

\begin{proposition}\label{notsosymm}
Suppose that $\cB$ is a double blocking set of size $3q-1$ in $\PG(2,q)$, $q\neq2$, where all points of $L_X$ and $L_Y$ are in $\cB$ except for the points $X_1$, $X_\infty$, $Y_1$ and $Y_\infty$. Then there is at most one nontrivial projectivity fixing $\cB$, and if there is one, it must correspond to a reflection in $G\simeq D_4$.
\end{proposition}
\begin{proof}
Let $\cS=\cB \setminus (L_X \cup L_Y)$ and let $\cL_O$ denote the set of lines through the origin different from $L_X$ and $L_Y$. 
Suppose that there are two nontrivial projectivities fixing $\cB$. Then either at least one of them or their product is a rotation in $G$, thus the subgroup generated by them must contain $\varphi:=\langle F^2\rangle=\left\langle\begin{pmatrix}1&0&0\\0&1&0\\1&1&-1\end{pmatrix}\right\rangle$, which thus also fixes $\cB$. It is easy to see that $\varphi^2$ is the identity and the fixpoints of $\varphi$ are $(0:0:1)$ and the points of the line $[1:1:-2]$. 
For any $\ell\in\cL_O$, $\varphi$ fixes $\ell$ and it follows that 
\begin{equation}
\ell\in\cL_O \Rightarrow \left|\ell\cap\cB\setminus\left((\ell\cap[1:1:-2])\cup\{(0:0:1)\}\right)\right|\equiv 0 \pmod 2. \label{even}
\end{equation}
Suppose now $2\mid q$ and recall that we may assume $q\geq 8$. Then $[1:1:-2]=[1:1:0]$ passes through $(0:0:1)$. 
Every line in $\cL_O\setminus\{[1:1:0]\}$ contains at least one, and thus by \eqref{even}, at least two points of $\cS$, whence $q+2=|\cS|\geq 2(q-2)$, a contradiction. 
Thus we may assume that $q$ is odd. By Corollary \ref{Osecants}, in $\cL_O$ there are either three trisecants (if $q\equiv 1 \pmod 3$) or one five-secant (if $q\equiv 0 \pmod 3$) to $\cB$ and the rest of the lines of $\cL_O$ are bisecants to $\cB$. As $[1:1:-2]\notin\cL_O$, it follows from \eqref{even} that for all $\ell\in\cL_O$, $|\ell\cap\cS| = 1\iff \ell\cap[1:1:-2]\in\cS$. 
Suppose now $3\mid q$. By Corollary \ref{Osecants}, there is one line in $\cL_O$ containing more than one point of $\cS$, hence $[1:1:-2]$ is a $q$-secant to $\cB$; thus $|\cB|\geq 3q$, a contradiction. 
Thus we may assume that $q\equiv 1\pmod 3$. Let the three trisecants to $\cB$ in $\cL_O$ be $\ell_1=[s:-1:0]$, $\ell_2=[\mu s:-1:0]$ and $\ell_3=[\mu^2s:-1:0]$, where $s,\mu\in\F_q^*$ and $\mu^2+\mu+1=0$ (cf.~Lemma \ref{trisecants}). Then $[1:1:-2]$ is a $(q-2)$-secant to $\cB$ with holes (i.e., points not in $\cB$) $H_i:=\ell_i\cap[1:1:-2]$, $i=1,2,3$.
\begin{figure}[!ht]
\psfrag{x}{$X_1$}\psfrag{xi}{$X_\infty$}\psfrag{y}{$Y_1$}\psfrag{yi}{$Y_\infty$}\psfrag{p1}{$P_1$}\psfrag{q1}{$Q_1$}\psfrag{r1}{$R_1$}\psfrag{p2}{$P_2$}\psfrag{q2}{$Q_2$}\psfrag{r2}{$R_2$}\psfrag{l1}{$\ell_1$}\psfrag{l2}{$\ell_2$}\psfrag{l3}{$\ell_3$}\psfrag{h1}{$H_1$}\psfrag{h2}{$H_2$}\psfrag{h3}{$H_3$}
\begin{center}
\includegraphics[width=0.95\textwidth]{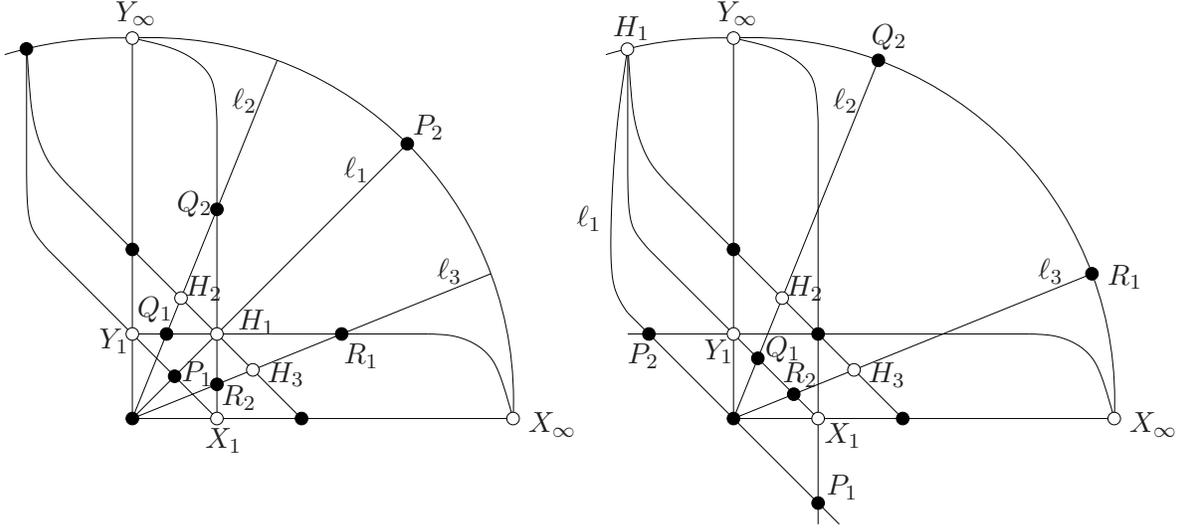}
\caption{Arrangements for $q\equiv1\pmod3$, $s^3=1$ (to the left) and $s^3=-1$ (to the right).\label{fig1}}
\end{center}
\end{figure}

Suppose $s^3=1$; we may assume $s=1$ (see Figure \ref{fig1}). Then $H_1=(1:1:1)$ cannot be in $\cS$, and $H_2=(2:2\mu:1+\mu)$, $H_3=(2:2\mu^2:1+\mu^2)$. 
As the lines $X_\infty Y_1=H_1X_\infty$ and $Y_\infty X_1=H_1Y_\infty$ must contain two points in $\cS$, we see that  $Q_1:=X_\infty Y_1\cap\ell_2=(1:\mu:1)$, $R_1:=X_\infty Y_1\cap\ell_3=(1:\mu^2:1)$, $Q_2:=Y_\infty X_1\cap\ell_2=(\mu^2:1:1)$ and $R_2:=Y_\infty X_1\cap\ell_3=(\mu:1:1)$ must be in $\cS$.  None of these four points are on $X_1Y_1=[1:1:-1]$, which is a $2$-secant to $\cS$; thus $P_1:=[1:1:-1]\cap\ell_1=(1:1:2)$ must be in $\cS$, as well as $P_2:=(1:1:2)^\varphi=(1:1:0)$. Let $\cS'=\{P_1,P_2,Q_1,Q_2,R_1,R_2\}$. Consider $Y_\infty H_2=[\mu+1:0:-2]$. This line must contain a point of $\cS'$ which, clearly, cannot be $Q_1$, $Q_2$, $R_2$ or $P_2$. As $R_1=(1:\mu^2:1)\in[\mu+1:0:-2]\iff \mu=1$ is not possible, we have $P_1=(1:1:2)\in[\mu+1:0:-2]$, that is, $\mu=3$. As $\mu^2+\mu+1=13=0$, this yields $13\mid q$. Consider now $Y_\infty H_3=[\mu^2+1:0:-2]=[5:0:-1]$. This line also must contain a point of $\cS'$, which clearly cannot be $P_2$, $Q_2$, $R_1$ or $R_2$. As neither $Q_1=(1:3:1)\in[5:0:-1]$ nor $P_1=(1:1:2)\in[5:0:-1]$ holds, we obtained a contradiction. Hence $s^3=1$ cannot hold.

Now, as $s^3\neq 1$, $(1:1:1)\in\cB$ must hold. Then, by Proposition \ref{relations}, $P_1:=(1:s^3:1)$ and $P_2:=(s^{-3}:1:1)$ are also in $\cB$. As these points are not on $[1:1:-2]$, $(1:s^3:1)\in\cB$ yields $s^3\in\{s,\mu s, \mu^2 s\}$, whence $s=-1$ may be assumed and $H_1=(1:-1:0)$, $H_2=(2:-2\mu:1-\mu)$, $H_3=(2:-2\mu^2:1-\mu^2)$ follow (see Figure \ref{fig1}). Note that $P_1=(1:-1:1)$ and $P_2=(-1:1:1)$ are in $\ell_1$. As $H_1=(1:-1:0)\notin\cB$, the two points of $\cB$ on the line $[1:1:-1]$ must be its intersection points with $\ell_2$ and $\ell_3$, namely, $Q_1:=(1:-\mu:1-\mu)$ and $R_1:=(1:-\mu^2:1-\mu^2)$. Their images under $\varphi$, $Q_2:=(1:-\mu:0)$ and $R_2:=(1:-\mu^2:0)$, are also in $\cB$. Let $\cS'=\{P_1,P_2,Q_1,Q_2,R_1,R_2\}$. 
Consider $Y_1H_2=[\mu+1:2:-2]$. This line must contain a point of $\cS'$ which, clearly, cannot be $Q_1$, $Q_2$ or $R_1$. As neither $R_2=(1:-\mu^2:0)\in[\mu+1:2:-2]\iff 3\mu^2=0$ nor $P_2=(-1:1:1)\in[\mu+1:2:-2]\iff \mu=-1$ is possible, we obtain that $P_1=(1:-1:1)\in[\mu+1:2:-2]$, from which $\mu=3$ and $13\mid q$ follow. Consider now $Y_1H_3=[\mu^2+1:2:-2]=[3:-2:2]$. This line also must contain a point of $\cS'$ which clearly cannot be $R_1$, $R_2$ or $Q_1$. But none of $P_1=(1:-1:1)$, $P_2=(-1:1:1)$ and $Q_2=(1:-3:0)$ lie on $[-3:2:-2]$, thus we end up with a final contradiction.
\end{proof}

If $q$ is not a prime, then there are double blocking sets in $\PG(2,q)$ that are much smaller than $3q$. Thus in this case not the size but the structure of such constructions may be of interest. At the end of this section, let us point out that minimality is not an issue in our case.

\begin{proposition}
Let $\cB$ be a blocking set in a projective plane of order $q>4$ of size $3q-1$ that admits two $(q-1)$-secants whose intersection point is in $\cB$. Then $\cB$ is minimal.
\end{proposition}
\begin{proof}
Let $\ell_1$ and $\ell_2$ be the two $(q-1)$-secants, $\{P_i,Q_i\}:=\ell_i\setminus\cB$, $i=1,2$. Considering the lines through $P_1$ and the points of $\cB$ on them, we see $q-1$ points on $\ell_1$, at least two points on $P_1P_2$ and $P_1Q_2$, and at least one point on each of the $q-2$ further lines $P_1P$, $P\in\ell_2\cap\cB$, $P\neq\ell_1\cap\ell_2$. This is at least $3q-1$ points altogether, thus equality must hold everywhere, whence we see that all points of $\cB\setminus \ell_1$ are essential for $\cB$. Repeating the argument with $P_2$ we get that the only point that could be superfluous is $\ell_1\cap\ell_2=:O$. In this case, looking around from $O$ we obtain $|\cB\setminus\{O\}|=3q-2\geq 2\cdot(q-2)+(q-1)\cdot2$, that is, $q\leq 4$.

\end{proof}

\section{Proof of Theorem \ref{Hillconjthm}}
\label{Hillconj}

Let $\cB$ be a double blocking set in $\PG(2,q)$ of size $3q-1$. Suppose that $\cB$ contains all the points of a line $\ell$. Consider $\ell$ as the line at infinity. Then $\cB\setminus\ell$ is a blocking set of $\AG(2,q)$ which, by the well-known result of Jamison \cite{Jamison} (and, independently, Brouwer and Schrijver \cite{BSch}), must have at least $2q-1$ points. Thus $|\cB|\geq 3q$, a contradiction. 

Recall that a double blocking set having a $q$-secant has at least $3q$ points. Suppose now that $\cB$ is obtained from a triangle by removing six of its points and adding five. Let us denote the three vertices of the triangle by $A$, $B$ and $C$. The 6 points removed from the sides will be called \emph{holes}. By the previous remarks, there must be two holes on each side. For $I\in\{A,B,C\}$, let $\ell_I$ denote the side of the triangle for which $I\notin\ell_I$, and we denote the holes on $\ell_I$ by $I_1$ and $I_2$. 
The 5 points of the blocking set not on the sides of the triangle will be called \emph{midpoints}, and we denote them by $P_1,\ldots,P_5$. We may assume that the three vertices of the triangle are the three base points $A=(0:0:1)$, $B=(1:0:0)$, and $C=(0:1:0)$.

The proof comes in two subsections depending on whether or not there are one or more triplets of the holes that are collinear.

\subsection{Holes in general position}

In this subsection we assume that the holes are in general position, that is, no three of them are collinear.
Let us denote the set of lines joining a vertex of the triangle with one of the holes of the opposite side by $\cL$.  
Note that $|\cL|=6$, thus there is a midpoint incident with at least two lines of $\cL$. Applying a suitable projectivity, we may move such a midpoint to $(1:1:1)$ without moving the triangle.

\begin{lemma}
\label{lemma1}
If the holes are in general position, then there is no midpoint incident with three lines of $\cL$.
\end{lemma}
\begin{proof}
Suppose the contrary and let $P_1$ be a midpoint incident with three lines of $\cL$. Then we may assume \[P_1=A_1A\cap B_1B\cap C_1C=(1:1:1)\] and hence $A_1=(1:1:0)$, $B_1=(0:1:1)$, and $C_1=(1:0:1)$ are holes. Note that if $q$ is even, then these are collinear (the line $[1:1:1]$ joins them), a contradiction. Hence we may assume that $q$ is odd.

\begin{figure}[!ht]
\footnotesize
\begin{center}
\psfrag{a}{$(0:0:1)\!=\!A$}\psfrag{b}{$B\!=\!(1:0:0)$}\psfrag{c}{$C\!=\!(0:1:0)$}\psfrag{a1}{$A_1\!=\!(1:1:0)$}\psfrag{b1}{$(0:1:1)\!=\!B_1$}\psfrag{c1}{$(1:0:1)\!=\!C_1$}\psfrag{a2}{$A_2\!=\!(1:m:0)$}\psfrag{b2}{$(0:y:1)\!=\!B_2$}\psfrag{c2}{$C_2\!=\!(x:0:1)$}\psfrag{p1}{$P_1\!=\!(1:1:1)$}
\hspace{-1.75cm}\includegraphics[height=4.6cm]{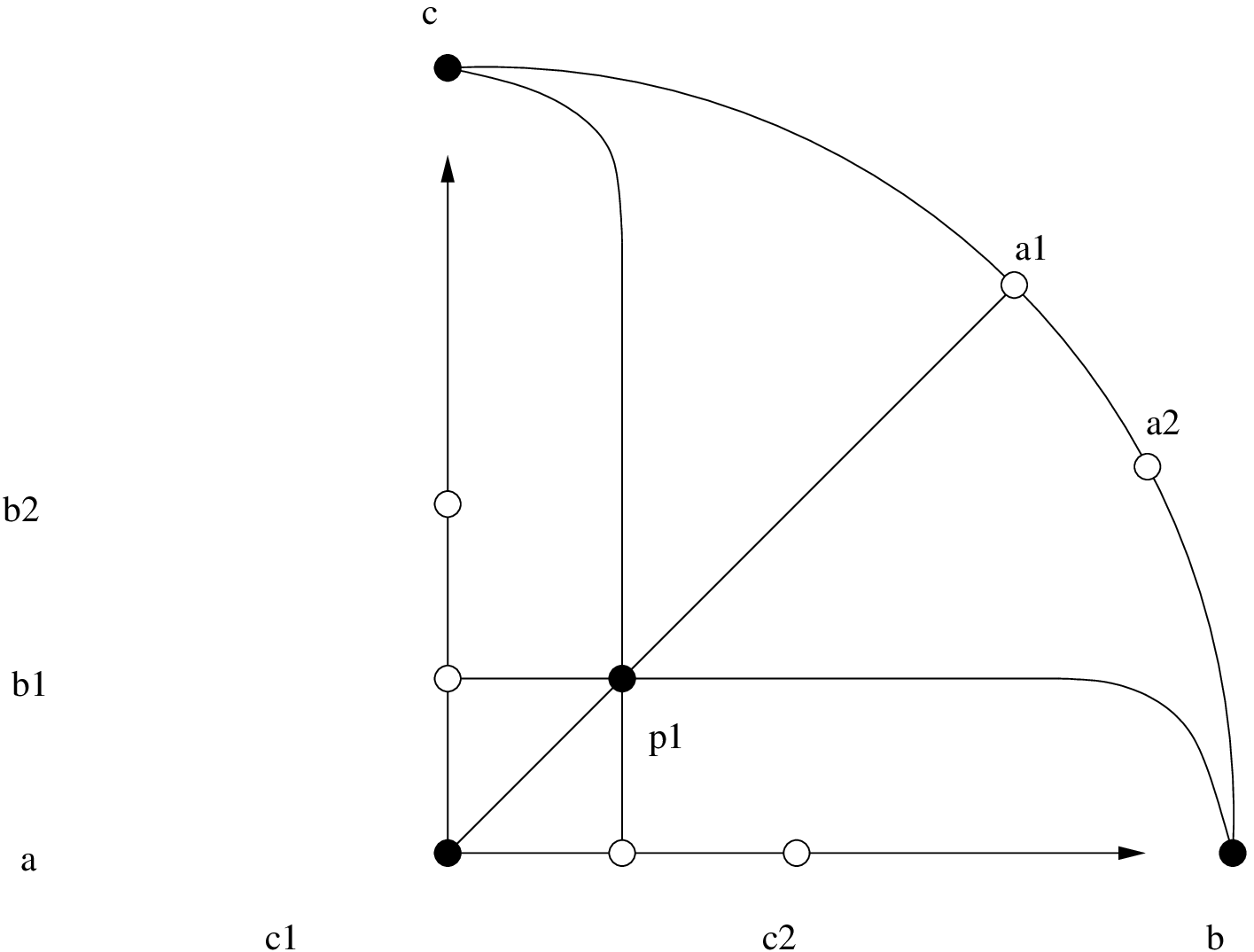}
\hspace{1.5cm}
\psfrag{a}{$(0:0:1)\!=\!A$}\psfrag{b}{$B\!=\!(1:0:0)$}\psfrag{c}{$C\!=\!(0:1:0)$}\psfrag{a1}{$A_1\!=\!(1:m:0)$}\psfrag{a12}{$\!=\!(1-x:1:0)$}\psfrag{b1}{$(0:1:1)\!=\!B_1$}\psfrag{c1}{$(1:0:1)\!=\!C_1$}\psfrag{a2}{$A_2\!=\!(1:m':0)$}\psfrag{a22}{$\!=\!(1:1-y:0)$}\psfrag{b2}{$(0:y:1)\!=\!B_2$}\psfrag{c2}{$C_2\!=\!(x:0:1)$}\psfrag{p1}{$P_1\!=\!(1:1:1)$}
\includegraphics[height=4.6cm]{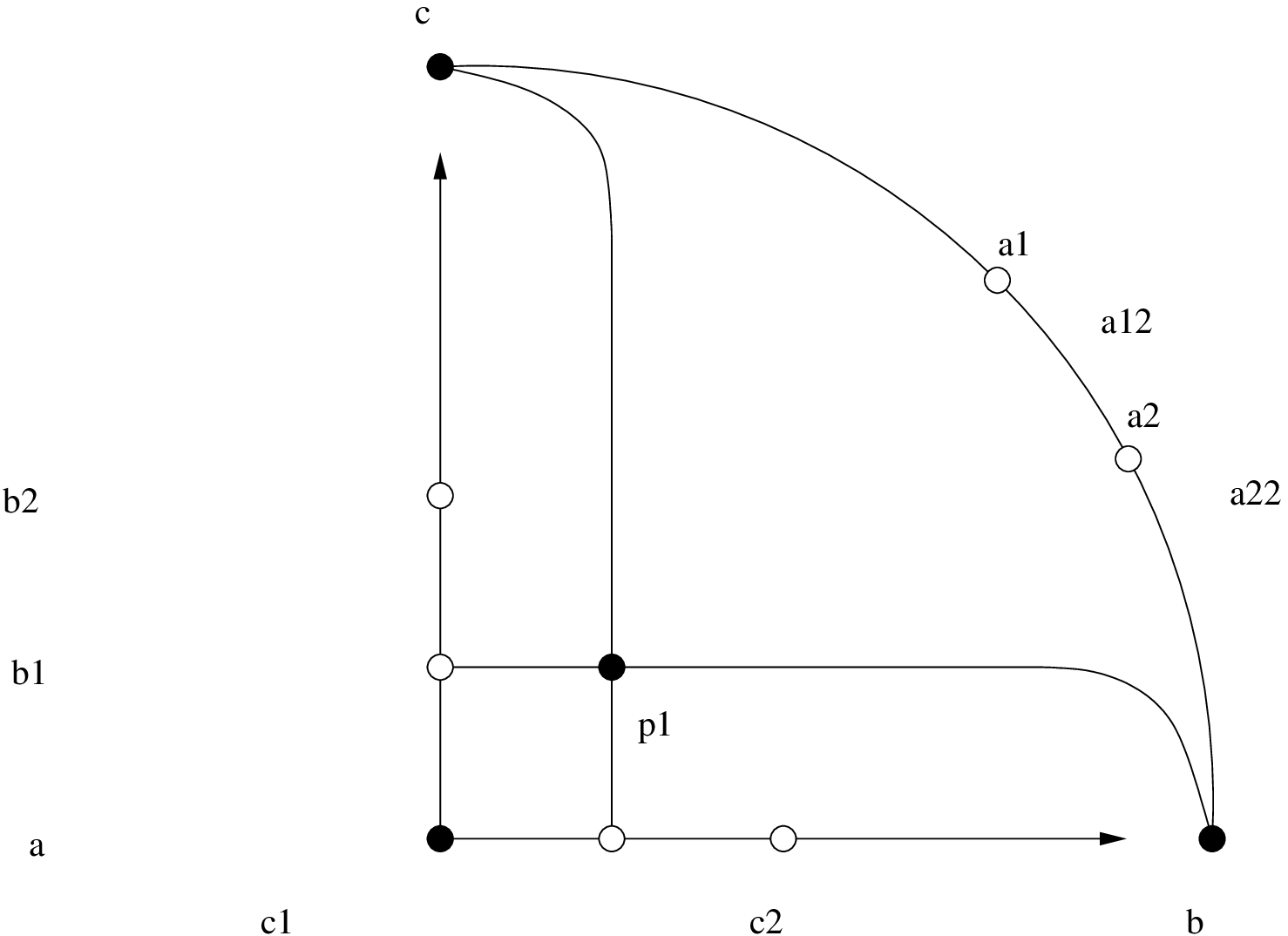}
\end{center}
\caption{The arrangement of the triangle and the holes. \label{fig2}}
\normalsize
\end{figure}

Denote the remaining holes by $C_2=(x:0:1)$, $B_2=(0:y:1)$ and $A_2=(1:m:0)$ (see Figure \ref{fig2} left). 
Looking at lines passing through $A_1$ we get that there is a midpoint incident with each of the lines
\[A_1B_2,\: A_1B_1,\: A_1C_1,\: A_1C_2. \]
The line $C_1B_1=[1:1:-1]$ is also incident with a midpoint, and this midpoint can lie neither on $A_1C_1$ nor on $A_1B_1$; also, it cannot be $P_1$. 
It follows that this midpoint is incident either with $A_1B_2$ or with $A_1C_2$. By interchanging the role of the $X$ and $Y$ axes, we may assume that it is incident with $A_1C_2=[-1:1:x]$ and hence \[P_2:=C_1B_1\cap A_1C_2=(1+x:1-x:2)\] is a midpoint.
Now consider the line $C_1B_2=[y:1:-y]$. It can be incident neither with $P_1$ nor with $P_2$, thus it is incident with a midpoint from one of the lines
$A_1B_2$, $A_1B_1$, $A_1C_1$. Clearly, this line has to be $A_1B_1=[1:-1:1]$ and hence \[P_3:=C_1B_2\cap A_1B_1=(y-1:2y:y+1)\] is a midpoint.
The line $C_1A_2=[-m:1:m]$ has to be incident with a midpoint $P_4$ which is clearly different from the previous midpoints. Also, $C_1A_2$ cannot have a common midpoint with the line $A_1C_1$, thus $P_4$ is incident with $A_1B_2=[1:-1:y]$ and hence \[P_4=C_1A_2\cap A_1B_2=(m+y:m+my:m-1).\]
Finally, there must be a midpoint incident with the line $B_1A_2=[m:-1:1]$ and this midpoint cannot coincide with the previous ones, thus it must be incident also with the line $C_1A_1=[-1:1:1]$, hence \[P_5:=B_1A_2\cap C_1A_1=(2:m+1:1-m)\] is a midpoint.

The line $B_1C_2=[1:x:-x]$ has to be incident with $P_4$ since all the other midpoints lie on different lines through $B_1$, thus
\begin{equation}
\label{eq1}
x+y+m+xym=0.
\end{equation}
It is easy to see that
\begin{itemize}
  \item $B_2A_2=[m:-1:y]$ is incident with at least one of $P_1$ and $P_2$,
  \item $B_2B=[0:-1:y]$ is incident with at least one of $P_5$ and $P_2$,
  \item $C_2B_2=[y:x:-xy]$ is incident with at least one of $P_1$ and $P_5$.
\end{itemize}
We distinguish two cases. In the first case we suppose that $B_2A_2$ is blocked by $P_1$, that is,
\begin{equation*}
y+m-1=0.
\end{equation*}
Then, as $P_1 \notin C_2B_2$, $P_5\in C_2B_2$ must hold, that is,
\begin{equation}
\label{eq3}
x+2y+mx-xy+xym=0.
\end{equation}
Also, since $P_5 \in C_2B_2$, we have $P_5\notin B_2B$ and hence $P_2 \in B_2B$, which means
\begin{equation*}
1-x=2y.
\end{equation*}
Then $y=1-m$ and $x=1-2y=2m-1$. Putting back these into \eqref{eq1} and \eqref{eq3} and subtracting these two equations from each other gives
\[2(m-1)(2m-1)=0,\]
a contradiction (as both $m=1$ and $2m-1=x=0$ are impossible).

In the second case we have $P_2\in B_2A_2$, that is,
\begin{equation}
\label{eq4}
mx=1-x-m-2y.
\end{equation}
Then $P_5 \in B_2B$ follows and hence
\begin{equation}
\label{eq5}
my=y-m-1.
\end{equation}
Finally, $P_1\in C_2B_2$; in other words,
\begin{equation}
\label{eq6}
xy=x+y.
\end{equation}
Combining \eqref{eq1}, \eqref{eq4}, \eqref{eq5}, and \eqref{eq6}, we obtain
\begin{align*}
0&=x+y+m+(x+y)m\\
 &= x+y+m+1-x-m-2y+y-m-1 = -m,
\end{align*}
 a contradiction.
\end{proof}

\begin{proposition}\label{genholes}
If the holes are in general position, then there is no 2-fold blocking set with the given properties.
\end{proposition}
\begin{proof}
As before, we may assume that $P_1=(1:1:1)$ is a midpoint and $B_1=(0:1:1)$, $C_1=(1:0:1)$ are holes.
Let $B_2=(0:y:1)$, $C_2=(x:0:1)$ denote the other two affine holes; note that $\{x,y\}\cap\{0,1\}=\emptyset$. 
Consider $A_1=(1:m:0)$, one of the holes at the line at infinity. There are different midpoints on each of the lines
$A_1C_2$, $A_1C_1$, $A_1B_2$, $A_1B_1$, $A_1A$, so $P_1$ is incident with one of these lines. Clearly, $P_1\notin A_1C_1$, $P_1\notin A_1B_1$ and, because of Lemma \ref{lemma1}, $P_1 \notin A_1A$, thus $P_1$ is incident with $A_1C_2=[-m:1:mx]$ or with $A_1B_2$. 
Now consider $A_2=(1:m':0)$, the other hole at the line at infinity (see Figure \ref{fig2} right). In the same way we obtain $P_1\in A_2C_2$ or $P_1\in A_2B_2=[m':-1:y]$. Thus either $P_1\in A_1C_2\cap A_2B_2$ or $P_1\in A_2C_2\cap A_1B_2$. With a suitable (re)labeling of $A_1$ and $A_2$, we obtain
\begin{eqnarray*}
P_1&=&B_1B\cap C_1C\cap A_1C_2 \cap A_2B_2,\\
A_1&=&(1-x:1:0),\\
A_2&=&(1:1-y:0).
\end{eqnarray*}
The line $C_1B_1=[1:1:-1]$ has to be incident with one of the midpoints. Recall that each of the midpoints is incident with exactly one of the lines
$A_1B_1$, $A_1B_2$, $A_1A$, $A_1C_2$, $A_1C_1$. As $P_1\notin C_1B_1$ and $P_1\in A_1C_2$, the midpoint on $C_1B_1$ is either on $A_1B_2=[1:x-1:y-xy]$ or on $A_1A$.
Similarly, each of the midpoints is incident with exactly one of the lines $A_2B_1$, $A_2B_2$, $A_2A$, $A_2C_2$, $A_2C_1$, and $C_1B_1$ shares a midpoint either with $A_2C_2$ or with $A_2A$.
Since $C_1B_1$ is incident with exactly one midpoint, $A_1A \cap C_1B_1$ and $A_2A \cap C_1B_1$ cannot be both midpoints and hence
one of $A_1B_2\cap C_1B_1$ and $A_2C_2 \cap C_1B_1$ is a midpoint. After possibly interchanging the role of the $X$ and $Y$ axes,
we may assume that \[P_2:=A_1B_2\cap C_1B_1=(1-x-y+xy:1+y-xy:2-x)\] is a midpoint.
Now take the line $B_1A_2=[y-1:1:-1]$; note that neither $P_1$ nor $P_2$ lies on it, thus it must contain a different midpoint $P_3$. Consider the lines $A_1B_1$, $A_1B_2$, $A_1A$, $A_1C_1$, $A_1C_2$. A similar argument as before shows that $P_3$ is either on $A_1A=[1:x-1:0]$ or on $A_1C_1=[1:x-1:-1]$. We distinguish two cases according to these two possibilities.

\texttt{Case I: $P_3=B_1A_2 \cap A_1A=(1-x:1:x+y-xy)$.} Then looking around from $B_1$ and $A_1$, we see that the remaining two midpoints must block $A_1B_1$, $B_1C_2=[1:x:-x]$ and $A_1C_1$, hence
\[P_4:=B_1C_2\cap A_1C_1=(2x-x^2:x-1:1)\] is also a midpoint.

Consider the lines $A_1B_1$ and $A_2C_1$. The first four midpoints cannot be incident with them, thus $P_5:=A_1B_1\cap A_2C_1$ is a midpoint. But then $P_5\notin A_2C_2$, hence $P_2\in A_2C_2$ must hold; therefore $P_2\notin A_2A$, so $P_4\in A_2A$ follows. The line $B_2C_1$ can be blocked only by $P_3$, which yields $P_3\notin B_2C_2$, and thus $P_5\in B_2C_2$; consequently, $P_5\notin B_2B$, hence $P_4\in B_2B=[0:1:-y]$. This gives $x=y+1$. On the other hand, $P_4\in A_2A=[y-1:1:0]$ and $P_3\in B_2C_1=[y:1:-y]$ give 
\begin{align*}
-1-x+x^2+2xy-x^2y&=0,\\
-1-y+2xy+y^2-xy^2&=0,
\end{align*}
respectively. Subtracting these two equations from each other yields $(x-1)(y-1)(x-y)=0$. As $x,y\neq1$, $x=y$ must hold, in contradiction with $x=y+1$. 

\texttt{Case II: $P_3=B_1A_2 \cap A_1C_1=(x-2:y-2:xy-x-y)$.} 
We see that the remaining two midpoints must block $A_1B_1$, $A_1A$ and $B_1C_2=[1:x:-x]$, hence
\[P_4:=A_1A\cap B_1C_2=(x-x^2:x:1)\] is also a midpoint.
As $A_1B_1$ is not blocked by the first four midpoints, $P_5\in A_1B_1$.
The line $A_2C_1$ can be blocked by $P_4$ and $P_5$ only; $A_2A$ and $A_2C_2$ by $P_2$ and $P_5$ only. Thus $P_5\in A_2C_1$ is impossible, since $P_2$ cannot block both $A_2A$ and $A_2C_2$; thus $P_4\in A_2C_1=[y-1:1:1-y]$ and hence $P_5\in B_2C_1$ and $P_3\in B_2C_2=[y:x:-xy]$ follow. Consequently, $B_2B=[0:-1:y]$ can be blocked only by $P_4$.

The incidence $P_4\in B_2B$ gives $x=y$. Then $P_4=(x-x^2:x:1)\in A_2C_1=[x-1:1:1-x]$ gives $x^3-2x^2=1-x$ and $P_3=(x-2:x-2:x^2-2x)\in B_2C_2=[x:x:-x^2]=[1:1:-x]$ gives $x^3-2x^2=2x-4$. It follows that $q$ must be odd. From $1-x=2x-4$ we get $x=5/3$; with this, $x^3-2x^2=1-x$ holds if and only if $p=7$. Consider now the lines $A_2A$, $A_2C_2$ and $C_2C$. These must be blocked by $P_2$ and $P_5$; consequently, either $P_2\in A_2C_2$ or $P_2=A_2A\cap C_2C$. Under $p=7$ and $x=y=5/3=4$, it is easy to see that neither $P_2=(2:3:5)\in A_2C_2=[3:1:2]$, nor $P_2\in A_2A=[4:1:0]$.
\end{proof}

By Proposition \ref{genholes}, we see that there must be at least one triplet among the holes that is collinear, which case is to be treated in the next subsection.

\subsection{Holes with collinear triplets}

With the general notation of the entire section, we assume this time that the holes $A_1$, $B_1$ and $C_1$ are collinear, and $\ell$ denotes the line joining them. 
As the collineation group of $\PG(2,q)$ is transitive on the quadruples of four lines in general position, we may assume that $\ell=[1:-1:1]$, $\ell_A=[0:0:1]$, $\ell_B=[1:0:0]$ and $\ell_C=[0:1:0]$. 
In this setting, for some $x,y,m\in\F_q$, we have
\begin{align*}
A&= (0:0:1), & B&= (1:0:0), & C&= (0:1:0), \\
A_1&= (1:1:0), & B_1&= (0:1:1), & C_1&=(-1:0:1), \\
A_2&= (1:m:0), & B_2&= (0:y:1), & C_2&= (x:0:1).
\end{align*}
Clearly, $x\notin\{0,-1\}$ and $\{y,m\}\cap\{0,1\}=\emptyset$. Note that with a suitable collineation, $\ell_A$, $\ell_B$ and $\ell_C$ can be arbitrarily permuted while fixing $\ell$.

\begin{lemma}\label{2collin}
If there are more than one collinear triplets among the holes, then these triplets have to be disjoint.
\end{lemma}
\begin{proof}
Suppose to the contrary that there is another collinear triplet among the holes which has a common point with $\{A_1,B_1,C_1\}$. We may assume that this triplet is $\{A_2,B_1,C_2\}$. Let $\ell'$ be the line joining these holes. Then both $\ell$ and $\ell'$ contain at least two midpoints and hence there is at most one midpoint $P_1$ which is not contained in $\ell\cup \ell'$. Since $A_2\in B_1C_2=[1:x:-x]$, we have $A_2=(-x:1:0)$.
It is easy to see that each of the lines $B_1B=[0:1:-1]$, $A_1C_2=[-1:1:x]$ and $C_1A_2=[1:x:1]$ must be blocked by $P_1$. 
Then $P_1=B_1B\cap A_1C_2=(x+1:1:1)$; furthermore, $P_1\in C_1A_2$ gives $2x=-2$, a contradiction unless $2\mid q$. 

Suppose now $2\mid q$. Let the midponits on $\ell$ be $P_2$ and $P_3$, and let $P_4$, $P_5$ be the midpoints on $\ell'$. 
Note that $A_1$, $B_2$ and $C_2$ cannot be collinear as, if they were, there should be two midpoints on the line joining them, but none of $P_2$, $P_3$, $P_4$, or $P_5$ can be on it.
Then $C_2C=[1:0:x]$ and $C_2B_2=[y:x:xy]$ can be blocked only by $P_2$ and $P_3$; thus, without loss of generality we may assume that $P_2 = C_2C\cap \ell=(x:x+1:1)$ and $P_3=C_2B_2\cap \ell=(x(y+1):y(x+1):x+y)$. 
Similarly as above, we can argue that $A_2$, $B_2$ and $C_1$ cannot be collinear. This yields $P_1\notin A_2B_2$. 
Since $A_2C_2\cap B_1B_2 = B_1$, $A_2$, $B_2$ and $C_2$ are not collinear, hence $P_3\notin A_2B_2$. Clearly, $P_4$ and $P_5$ are not on $A_2B_2$; thus $P_2\in A_2B_2=[1:x:xy]$ must hold. This yields $x(x+y)=0$, thus $x=y$; hence $P_3=(1:1:0)=A_1$, a contradiction.
\end{proof}

Let us call $A_1$, $B_1$ and $C_1$ `collinear holes' in the sequel. The line $\ell$ must contain two midpoints; let us denote the other three midpoints by $P_1$, $P_2$ and $P_3$, and those two on $\ell$ by $P_4$ and $P_5$. 
Consider a collinear hole, say, $A_1$. Then the lines $A_1A$, $A_1B_2$ and $A_1C_2$ are pairwise distinct by Lemma \ref{2collin}, and they must be blocked by $P_1$, $P_2$ and $P_3$ (in some order). Hence, using the same observation for all the three collinear holes, we see that for $i=1,2,3$, $P_i$ is incident with exactly one line of each row in the following table, and each of the nine lines is incident with exactly one of $P_1$, $P_2$ and $P_3$:
\begin{align*}
A_1A&= [1:-1:0], & A_1B_2&= [1:-1:y],   & A_1C_2&= [1:-1:-x], \\
B_1B&= [0:1:-1], & B_1C_2&= [1:x:-x], & B_1A_2&= [m:-1:1], \\
C_1C&= [1:0:1],  & C_1A_2&= [m:-1:m], & C_1B_2&= [y:-1:y].
\end{align*}
We will frequently use these observations and coordinates without explicitly referring to them.


\begin{proposition}\label{twoofL}
None of $A_1A\cap B_1B$, $B_1B\cap C_1C$, $C_1C\cap A_1A$ can be a midpoint. 
\end{proposition}
\begin{proof}
Suppose the contrary. Without loss of generality we may assume that $P_1=A_1A \cap C_1C=(-1:-1:1)$ is a midpoint. There must be a midpoint on $A_1B_2$, which cannot be on $C_1C$ or $C_1B_2$, so (with suitable relabeling) $P_2=A_1B_2\cap C_1A_2=(y-m:m(y-1):m-1)$ is a midpoint and, similarly, $P_3=A_1C_2\cap C_1B_2=(x+y:y(x+1):1-y)$ is also midpoint.
These three must block the lines $B_1B$, $B_1C_2$, $B_1A_2$. Note that $P_2\notin B_1A_2$ and $P_3\notin B_1C_2$.\\
\texttt{Case I: $P_1\in B_1B$.} That is, $(-1:-1:1)\in[0:1:-1]$, which happens if and only if $q$ is even. This immediately leads to $P_2=A_1B_2\cap B_1C_2\cap C_1A_2$ and $P_3=A_1C_2\cap B_1A_2\cap C_1B_2$. The arising equations yield \begin{equation}
m+x+y=mxy \mbox{\: and \:} yx+mx+ym=1. \label{eqsymm}
\end{equation}
From these we obtain $m(1+xy)=(x+y)$ and $m(x+y)=1+xy$,
hence $m^2(1+xy)=(x+y)m=1+xy$, so either $m=1$ or $xy=1$. As the first option is forbidden, applying symmetry arguments in \eqref{eqsymm} we obtain $xy=xm=ym=1$, whence $xym=x=y=m=1$ follows, a contradiction.\\
\texttt{Case II: $P_1\in B_1C_2$.} Then $q$ is odd and $x=-\frac12$; furthermore, $P_2\in B_1B$ and $P_3\in B_1A_2$. These give $m=1/(2-y)$ and $m(x+y)-(x+2)y+1=0$, which lead to $3(y-1)^2=0$, thus either $3\mid q$, which is impossible by Corollary \ref{0mod3}, or $y=1$, a contradiction.\\
\texttt{Case III: $P_1\in B_1A_2$.} Then $m=2$, $P_3\in B_1B$ and $P_2\in B_1C_2$. These give $x=(1-2y)/y$ and $2xy-3x+y-2=0$, which lead to $3(y-1)^2=0$, but this is still impossible.
\end{proof}

By Proposition \ref{twoofL}, we may assume in the sequel that $P_1\in A_1A$, $P_2\in B_1B$, $P_3\in C_1C$.

\begin{proposition}\label{case2}
None of $A_1A\cap B_1A_2$, $A_1A\cap C_1A_2$, $B_1B\cap A_1B_2$, $B_1B\cap C_1B_2$, $C_1C\cap A_1C_2$, $C_1C\cap B_1C_2$ can be a midpoint.
\end{proposition}
\begin{proof}
Suppose the contrary. Applying a suitable collineation permuting $\ell_A$, $\ell_B$ and $\ell_C$ (recall that we may permute the letters $A$, $B$ and $C$ in an arbitrary fashion), we may assume that $A_1A\cap B_1A_2=P_1$ is a midpoint. Then, necessarily, $B_1C_2$ and thus $A_1B_2$ are blocked by $P_3$ and, similarly, $C_1A_2$ can only be blocked by $P_2$; summing up, we get $P_1=A_1A\cap B_1A_2\cap C_1B_2$, $P_2=B_1B\cap C_1A_2\cap A_1C_2$, and $P_3=C_1C\cap A_1B_2\cap B_1C_2$. From these we obtain $P_1=(1:1:1-m)$ and $y(2-m)=1$, $P_2=(x+1:1:1)$ and $(x+2)m=1$, and $P_3=(-1:y-1:1)$ and $x(y-2)=1$. Substituting $m=1/(x+2)$ and $y=2 + 1/x$ into $y(2-m)=1$ leads to $3(x+1)^2=0$, thus 
either $3\mid q$, which is yet again impossible by Corollary \ref{0mod3}, or $x=-1$, a contradiction.
\end{proof}

It follows easily from Propositions \ref{twoofL} and \ref{case2} that the only possibility left is $P_1=A_1A\cap B_1C_2\cap C_1B_2$, $P_2=B_1B\cap C_1A_2\cap A_1C_2$, $P_3=C_1C\cap A_1B_2\cap B_1A_2$. From these, simple calculations yield that $q$ is odd, $P_1=(y:y:1-y)$, $P_2=(x+1:1:1)$, $P_3=(-1:y-1:1)$, and $2xy+y-x=0$, $mx+2m-1=0$, $2-m-y=0$. Substituting $m=2-y$ into the second one we obtain $y=(2x+3)/(x+2)$ which, after substituting it into the first one, gives $3(x+1)^2=0$. This contradicts either Corollary \ref{0mod3} or $x\neq -1$.
With this, we have finished the proof of Theorem \ref{Hillconjthm}.

\section{Constructions of double blocking sets of size $3q-1$}
\label{constructions}

With the help of a standard PC and the MIP solvers \cite{glpk,gurobi}, we found some constructions of double blocking sets of size $3q-1$ in $\PG(2,q)$, $13\leq q\leq 43$, $q\not\equiv2\pmod 3$. We were looking for examples that admit two $(q-1)$-secants. Applying a suitable collineation, we may assume that these long secants are the $X$ and $Y$ axes, and the holes on them are $(1:0:1)$, $(1:0:0)$, $(0:1:1)$, and $(0:1:0)$. Hence we only give the coordinates of the remaining $q+2$ points. As an additional information, we also give the distribution of the secant lengths; to this end, let $n_t$ denote the number of $t$-secants with respect to the pointset under consideration. 
Sometimes, in order to fasten the calculations, we assumed the pointset to be $X$-$Y$ symmetric; that is, the collineation interchanging $(1:0:1)$ with $(0:1:1)$ and $(1:0:0)$ with $(0:1:0)$ (denoted by $T$ in Section \ref{structure}) should fix the double blocking set. Note that by Proposition \ref{notsosymm}, a construction admitting a nontrivial symmetry cannot have another nontrivial symmetry, and so it can be transformed into one which is $X$-$Y$ symmetric.
We also made use of the other structural properties derived in Section \ref{structure}, which remarkably reduced the necessary computer time. Unless we explicitly state differently in the notes, the trisecants through the origin in case of $q\equiv 1\pmod 3$ have slopes $-1$, $-\mu$ and $-\mu^2$ (where $\mu^2+\mu+1=0$; c.f.~Lemma \ref{trisecants} and Corollary \ref{Osecants}); in other words, the parameter $s$ of the example is $-1$. Note that for an example admitting the $X$-$Y$ symmetry, $s=\pm1$ necessarily holds as the symmetry implies $\{m,m'\}=\{1/m,1/m'\}$, and hence $-s^3=mm'=\pm1$ (c.f.~Proposition \ref{relations}). In many, but not all, of our examples $m=\mu$, $m'=\mu^2$.

\subsection{$q=13$}

\begin{tabular}{rl}
Points: & (points on the $X$ and $Y$ axes are not displayed)\\
& \begin{tabular}{ccccc}
$(1:1:1)$ & $(1:12:1)$ & $(2:8:1)$ & $(3:7:1)$ & $(4:3:1)$\\
$(5:9:1)$ & $(6:10:1)$ & $(7:4:1)$ & $(8:2:1)$ & $(9:5:1)$\\
$(10:11:1)$ & $(11:6:1)$ & $(12:1:1)$ & $(1:3:0)$ & $(1:9:0)$
\end{tabular}
\end{tabular}

\begin{tabular}{rl}
Secant distribution: & (the number $n_t$ of $t$-secants is present iff $t\geq 3$ and $n_t>0$)\\
& \begin{tabular}{r|ccccccc}
$t$ &
12 & 8 & 7 & 6 & 5 & 4 & 3 \\
\hline
$n_t$ &
2 & 1 & 1 & 4 & 10 & 19 & 51
\end{tabular}
\end{tabular}

\begin{tabular}{rl}
Notes: & The third roots of unity are $1,3,9$.\\
& Up to projective equivalence, this is the only example admitting two $(q-1)$-secants.\\
& There is no example having a symmetry.\\
& This example is different from the one published in \cite{BKW}, as the longest secants to that\\& one are $10$-secants.
\end{tabular}

\subsection{$q=16$}

Let $\omega$ be a primitive element of $\F_{16}$ which has minimal polynomial $x^4+x+1$ over $\F_2$. 

\begin{tabular}{rl}
Points: & (points on the $X$ and $Y$ axes are not displayed)\\
& \begin{tabular}{ccccc}
$(1:\omega^{14}:1)$ & $(1:\omega^7:1)$ & $(\omega:\omega^{11}:1)$ & $(\omega^2:\omega^3:1)$ & $(\omega^3:\omega^5:1)$\\
$(\omega^4:\omega:1)$ & $(\omega^5:\omega^{12}:1)$ & $(\omega^6:\omega^9:1)$ & $(\omega^7:\omega^{13}:1)$ & $(\omega^8:\omega^2:1)$\\
$(\omega^9:\omega^6:1)$ & $(\omega^{10}:\omega^{10}:1)$ & $(\omega^{11}:1:1)$ & $(\omega^{12}:\omega^8:1)$ & $(\omega^{13}:1:1)$ \\
$(\omega^{14}:\omega^4:1)$ & $(1:\omega^8:0)$ &  $(1:\omega^{13}:0)$
\end{tabular}
\end{tabular}

\begin{tabular}{rl}
Secant distribution: & (the number $n_t$ of $t$-secants is present iff $t\geq 3$ and $n_t>0$)\\
& \begin{tabular}{r|ccccccc}
$t$ &
15 & 9 & 7 & 6 & 5 & 4 & 3 \\
\hline
$n_t$ &
2 & 1 & 1 & 3 & 20 & 37 & 69
\end{tabular}
\end{tabular}

\begin{tabular}{rl}
Notes: & The third roots of unity are $1,\omega^{5},\omega^{10}$.\\
& The trisecants through the origin have slopes $\omega^2$, $\omega^7$ and $\omega^{12}$ (so $s=\omega^6$).\\
& There is no example where $[1:1:0]$ is a triscant (i.e., $s=-1=1$ is impossible).\\
& Therefore, there is no example admitting a symmetry; \\
& and there is no example where $(1:\mu:0)$ and $(1:\mu^2:0)$ are both in $\cB$.
\end{tabular}

\subsection{$q=19$}

\subsubsection*{First example:}

\begin{tabular}{rl}
Points: & (points on the $X$ and $Y$ axes are not displayed)\\
& \begin{tabular}{cccccc}
$(1:4:1)$ & $(1:14:1)$ & $(2:18:1)$ & $(3:5:1)$ & $(6:13:1)$ & $(7:17:1)$ \\
$(4:1:1)$ & $(14:1:1)$ & $(18:2:1)$ & $(5:3:1)$ & $(13:6:1)$ & $(17:7:1)$ \\
$(8:16:1)$ & $(9:15:1)$ & $(10:11:1)$ & $(12:12:1)$ &  &  \\
$(16:8:1)$ & $(15:9:1)$ & $(11:10:1)$ &             & $(1:7:0)$ & $(1:11:0)$ \\
\end{tabular}
\end{tabular}

\begin{tabular}{rl}
Secant distribution: & (the number $n_t$ of $t$-secants is present iff $t\geq 3$ and $n_t>0$)\\
& \begin{tabular}{r|ccccccc}
$t$ &
18 & 11 & 7 & 6 &  5  & 4  &  3  \\
\hline
$n_t$ & 
2  & 1  &  2 & 4 & 22 & 57 &  111
\end{tabular}
\end{tabular}

\subsubsection*{Second example:}

\begin{tabular}{rl}
Points: & (points on the $X$ and $Y$ axes are not displayed)\\
& \begin{tabular}{cccccc}
$(1:4:1)$ & $(1:14:1)$ & $(2:12:1)$ & $(3:17:1)$ & $(5:5:1)$ & $(6:7:1)$ \\
$(4:1:1)$ & $(14:1:1)$ & $(12:2:1)$ & $(17:3:1)$ &           & $(7:6:1)$ \\
$(8:11:1)$ & $(9:13:1)$ & $(10:15:1)$ & $(16:18:1)$ &  &  \\
$(11:8:1)$ & $(13:9:1)$ & $(15:10:1)$ & $(18:16:1)$ & $(1:2:0)$ & $(1:10:0)$ \\
\end{tabular}
\end{tabular}

\begin{tabular}{rl}
Secant distribution: & (the number $n_t$ of $t$-secants is present iff $t\geq 3$ and $n_t>0$)\\
& \begin{tabular}{r|ccccccc}
$t$ &
18 & 8 & 7 & 6 &  5  & 4  &  3  \\
\hline
$n_t$ & 
2  & 2 & 2 & 5 & 28 & 39 &  128
\end{tabular}
\end{tabular}

\begin{tabular}{rl}
Notes: & The third roots of unity are $1,7,11$.\\
& Up to projective equivalence, there are no other examples;\\
& thus all examples admit a symmetry.
\end{tabular}

\subsection{$q=25$}

Let $\omega$ be a primitive element of $\F_{25}$ which has minimal polynomial $x^2-x-1$ over $\F_5$.

\begin{tabular}{rl}
Points: & (points on the $X$ and $Y$ axes are not displayed)\\
& \begin{tabular}{ccccc}
$(1:\omega^{4}:1)$ & $(1:\omega^{8}:1)$ & $(2:\omega^{10}:1)$ & $(3:\omega^{11}:1)$ & $(4:\omega^{2}:1)$ \\
$(\omega^{4}:1:1)$ & $(\omega^{8}:1:1)$ & $(\omega^{10}:2:1)$ & $(\omega^{11}:3:1)$ & $(\omega^{2}:4:1)$ \\
$(\omega:\omega^{13}:1)$ & $(\omega^{3}:\omega^{21}:1)$ & $(\omega^{5}:\omega^{14}:1)$ & $(\omega^{7}:\omega^{9}:1)$ & $(\omega^{15}:\omega^{20}:1)$ \\
$(\omega^{13}:\omega:1)$ & $(\omega^{21}:\omega^{3}:1)$ & $(\omega^{14}:\omega^{5}:1)$ & $(\omega^{9}:\omega^{7}:1)$ & $(\omega^{20}:\omega^{15}:1)$ \\
$(\omega^{16}:\omega^{17}:1)$ & $(\omega^{19}:\omega^{22}:1)$ & $(\omega^{23}:\omega^{23}:1)$ & & \\
$(\omega^{17}:\omega^{16}:1)$ & $(\omega^{22}:\omega^{19}:1)$ & 							  & $(1:\omega^{11}:0)$ &  $(1:\omega^{13}:0)$\\
\end{tabular}
\end{tabular}

\begin{tabular}{rl}
Secant distribution: & (the number $n_t$ of $t$-secants is present iff $t\geq 3$ and $n_t>0$)\\
& \begin{tabular}{r|ccccccc}
$t$ &
24 & 9 & 7 & 6  & 5 & 4  & 3 \\
\hline
$n_t$ &
2 & 1 & 6 & 15 & 24 & 101 & 207
\end{tabular}
\end{tabular}

\begin{tabular}{rl}
Notes: & The third roots of unity are $1,\omega^{8},\omega^{16}$.
\end{tabular}

\subsection{$q=27$}

Let $\omega$ be a primitive element of $\F_{27}$ which has minimal polynomial $x^3-x+1$ over $\F_3$.

\begin{tabular}{rl}
Points: & (points on the $X$ and $Y$ axes are not displayed)\\
& \begin{tabular}{ccccc}
$(1:\omega:1)$ & $(1:\omega^{12}:1)$ & $(2:\omega^{3}:1)$ & $(\omega^{2}:\omega^{25}:1)$ & $(\omega^{4}:\omega^{17}:1)$ \\
$(\omega:1:1)$ & $(\omega^{12}:1:1)$ & $(\omega^{3}:2:1)$ & $(\omega^{25}:\omega^{2}:1)$ & $(\omega^{17}:\omega^{4}:1)$ \\
$(\omega^{5}:\omega^{10}:1)$ & $(\omega^{6}:\omega^{24}:1)$ & $(\omega^{7}:\omega^{7}:1)$ & $(\omega^{8}:\omega^{21}:1)$ & $(\omega^{9}:\omega^{20}:1)$ \\
$(\omega^{10}:\omega^{5}:1)$ & $(\omega^{24}:\omega^{6}:1)$ & 							  & $(\omega^{21}:\omega^{8}:1)$ & $(\omega^{20}:\omega^{9}:1)$ \\
$(\omega^{11}:\omega^{18}:1)$ & $(\omega^{14}:\omega^{23}:1)$ & $(\omega^{15}:\omega^{19}:1)$ & $(\omega^{16}:\omega^{22}:1)$ & \\
$(\omega^{18}:\omega^{11}:1)$ & $(\omega^{23}:\omega^{14}:1)$ & $(\omega^{19}:\omega^{15}:1)$ & $(\omega^{22}:\omega^{16}:1)$ & \\
$(1:\omega^{2}:0)$ & $(1:\omega^{24}:0)$ & & & \\
\end{tabular}
\end{tabular}

\begin{tabular}{rl}
Secant distribution: & (the number $n_t$ of $t$-secants is present iff $t\geq 3$ and $n_t>0$)\\
& \begin{tabular}{r|cccccc}
$t$ &
26 & 7 & 6  & 5 & 4    & 3 \\
\hline
$n_t$ &
2 &  2 & 15 & 57 & 124 & 195
\end{tabular}
\end{tabular}

\begin{tabular}{rl}
Notes: & the five-secant through the origin has slope $-1$.\\
\end{tabular}

\subsection{$q=31$}

\subsubsection*{First example:}

\begin{tabular}{rl}
Points: & (points on the $X$ and $Y$ axes are not displayed)\\
& \begin{tabular}{cccccc}
$(1:1:1)$ & $(1:30:1)$ & $(2:12:1)$ & $(3:11:1)$ & $(4:6:1)$ & $(5:9:1)$ \\
          & $(30:1:1)$ & $(12:2:1)$ & $(11:3:1)$ & $(6:4:1)$ & $(9:5:1)$ \\
$(7:19:1)$ & $(8:13:1)$ & $(10:26:1)$ & $(14:23:1)$ & $(15:17:1)$ & $(16:22:1)$ \\
$(19:7:1)$ & $(13:8:1)$ & $(26:10:1)$ & $(23:14:1)$ & $(17:15:1)$ & $(22:16:1)$ \\
$(18:21:1)$ & $(20:25:1)$ & $(24:29:1)$ & $(27:28:1)$ & 		  &				\\
$(21:18:1)$ & $(25:20:1)$ & $(29:24:1)$ & $(28:27:1)$ & $(1:5:0)$ & $(1:25:0)$ \\
\end{tabular}
\end{tabular}

\begin{tabular}{rl}
Secant distribution: & (the number $n_t$ of $t$-secants is present iff $t\geq 3$ and $n_t>0$)\\
& \begin{tabular}{r|cccccccc}
$t$ &
30 & 10 & 8 & 7 &  6 &  5 &  4  &   3 \\
\hline
$n_t$ &
2  & 1 & 4  & 4 & 12 & 58 & 147 &  334
\end{tabular}
\end{tabular}

\subsubsection*{Second example:}

\begin{tabular}{rl}
Points: & (points on the $X$ and $Y$ axes are not displayed)\\
& \begin{tabular}{cccccc}
$(1:9:1)$ & $(1:24:1)$ & $(2:10:1)$ & $(3:29:1)$ & $(4:17:1)$ & $(5:6:1)$ \\
$(9:1:1)$ & $(24:1:1)$ & $(10:2:1)$ & $(29:3:1)$ & $(17:4:1)$ & $(6:5:1)$ \\
$(7:21:1)$ & $(8:16:1)$ & $(11:20:1)$ & $(12:28:1)$ & $(13:30:1)$ & $(14:25:1)$ \\
$(21:7:1)$ & $(16:8:1)$ & $(20:11:1)$ & $(28:12:1)$ & $(30:13:1)$ & $(25:14:1)$ \\
$(15:26:1)$ & $(18:22:1)$ & $(19:23:1)$ & $(27:27:1)$ & 		  &				\\
$(26:15:1)$ & $(22:18:1)$ & $(23:19:1)$ &             & $(1:6:0)$ & $(1:26:0)$ \\
\end{tabular}
\end{tabular}

\begin{tabular}{rl}
Secant distribution: & (the number $n_t$ of $t$-secants is present iff $t\geq 3$ and $n_t>0$)\\
& \begin{tabular}{r|cccccccc}
$t$ &
30 & 8 & 7 &  6 &  5 &  4  &   3 \\
\hline
$n_t$ &
2  & 2 & 2 & 24 & 58 & 153 & 304
\end{tabular}
\end{tabular}

\subsubsection*{Third example:}

\begin{tabular}{rl}
Points: & (points on the $X$ and $Y$ axes are not displayed)\\
& \begin{tabular}{cccccc}
$(1:21:1)$ & $(1:28:1)$ & $(2:30:1)$ & $(3:3:1)$ & $(4:27:1)$ & $(5:23:1)$ \\
$(21:1:1)$ & $(28:1:1)$ & $(30:2:1)$ &           & $(27:4:1)$ & $(23:5:1)$ \\
$(6:17:1)$ & $(7:11:1)$ & $(8:29:1)$ & $(9:14:1)$ & $(10:16:1)$ & $(12:22:1)$ \\
$(17:6:1)$ & $(11:7:1)$ & $(29:8:1)$ & $(14:9:1)$ & $(16:10:1)$ & $(22:12:1)$ \\
$(13:26:1)$ & $(15:25:1)$ & $(18:24:1)$ & $(19:20:1)$ & 		  &				\\
$(26:13:1)$ & $(25:15:1)$ & $(24:18:1)$ & $(20:19:1)$ & $(1:6:0)$ & $(1:26:0)$ \\
\end{tabular}
\end{tabular}

\begin{tabular}{rl}
Secant distribution: & (the number $n_t$ of $t$-secants is present iff $t\geq 3$ and $n_t>0$)\\
& \begin{tabular}{r|cccccccc}
$t$ &
30 & 7 &  6 &  5 &  4  &   3 \\
\hline
$n_t$ &
2  & 4 & 22 & 65 & 154 &  291
\end{tabular}
\end{tabular}

\begin{tabular}{rl}
Notes: & The third roots of unity are $1,5,25$.\\
\end{tabular}

\subsection{$q=37$}

\begin{tabular}{rl}
Points: & (points on the $X$ and $Y$ axes are not displayed)\\
& \begin{tabular}{cccccc}
$(1:3:1)$ & $(1:12:1)$ & $(2:2:1)$ & $(4:5:1)$ & $(6:32:1)$ & $(7:19:1)$ \\
$(3:1:1)$ & $(12:1:1)$ &         & $(5:4:1)$ & $(32:6:1)$ & $(19:7:1)$ \\
$(8:29:1)$ & $(9:25:1)$ & $(10:24:1)$ & $(11:35:1)$ & $(13:15:1)$ & $(14:16:1)$ \\
$(29:8:1)$ & $(25:9:1)$ & $(24:10:1)$ & $(35:11:1)$ & $(15:13:1)$ & $(16:14:1)$ \\
$(17:33:1)$ & $(18:28:1)$ & $(20:22:1)$ & $(21:27:1)$ & $(23:30:1)$ & $(26:34:1)$ \\
$(33:17:1)$ & $(28:18:1)$ & $(22:20:1)$ & $(27:21:1)$ & $(30:23:1)$ & $(34:26:1)$ \\
$(31:36:1)$ & $(36:31:1)$ &             &             & $(1:10:0)$  & $(1:26:0)$ \\
\end{tabular}
\end{tabular}

\begin{tabular}{rl}
Secant distribution: & (the number $n_t$ of $t$-secants is present iff $t\geq 3$ and $n_t>0$)\\
& \begin{tabular}{r|ccccccc}
$t$ &
36 & 8 & 7 &  6 &  5 &  4  &   3 \\
\hline
$n_t$ &
2  & 3 & 6 & 27 & 79 & 230 &  445
\end{tabular}
\end{tabular}

\begin{tabular}{rl}
Notes: & The third roots of unity are $1,10,26$.\\
\end{tabular}

\subsection{$q=43$}

\begin{tabular}{rl}
Points: & (points on the $X$ and $Y$ axes are not displayed)\\
& \begin{tabular}{cccccc}
$(1:18:1)$ & $(1:31:1)$ & $(2:39:1)$ & $(3:40:1)$ & $(4:36:1)$ & $(5:28:1)$ \\
$(18:1:1)$ & $(31:1:1)$ & $(39:2:1)$ & $(40:3:1)$ & $(36:4:1)$ & $(28:5:1)$ \\
$(6:27:1)$ & $(7:16:1)$ & $(8:32:1)$ & $(9:15:1)$ & $(10:26:1)$ & $(11:33:1)$ \\
$(27:6:1)$ & $(16:7:1)$ & $(32:8:1)$ & $(15:9:1)$ & $(26:10:1)$ & $(33:11:1)$ \\
$(12:21:1)$ & $(13:23:1)$ & $(14:22:1)$ & $(17:34:1)$ & $(19:19:1)$ & $(20:38:1)$ \\
$(21:12:1)$ & $(23:13:1)$ & $(22:14:1)$ & $(34:17:1)$ &             & $(38:20:1)$ \\
$(24:25:1)$ & $(29:41:1)$ & $(30:37:1)$ & $(35:42:1)$ &  &  \\
$(25:24:1)$ & $(41:29:1)$ & $(37:30:1)$ & $(42:35:1)$ & $(1:6:0)$ & $(1:36:0)$\\
\end{tabular}
\end{tabular}

\begin{tabular}{rl}
Secant distribution: & (the number $n_t$ of $t$-secants is present iff $t\geq 3$ and $n_t>0$)\\
& \begin{tabular}{r|ccccccc}
$t$ &
42 & 8 & 7 &  6 &  5  &  4  &   3 \\
\hline
$n_t$ &
2  & 4 & 8 & 26 & 122 & 321 &  590
\end{tabular}
\end{tabular}

\begin{tabular}{rl}
Notes: & The third roots of unity are $1,6,36$.\\
\end{tabular}

\section*{Acknowledgement}

We are thankful to Aart Blokhuis for his idea that helped finding the first example for $q=31$.

\end{document}